\documentclass[12pt,reqno]{amsart}
\usepackage{amsfonts}
\usepackage{a4wide}
\usepackage{graphicx}
\numberwithin{equation}{section}

\newtheorem{theorem}{Theorem}[section]

\newtheorem{lemma}[theorem]{Lemma}

\theoremstyle{definition}

\theoremstyle{remark}

\newcommand{\be}{\begin{equation}}
\newcommand{\ee}{\end{equation}}

\begin{document}

\title[ stability for gDNLS in the endpoint case
 ]
{ Orbital stability of solitary waves for  generalized derivative nonlinear Schr\"odinger equations in the endpoint case}

 \author{Qing Guo
}

\address{College of Science, Minzu University of China, Beijing 100081, China}
\email{guoqing0117@163.com}

\begin{abstract}
We consider the following generalized derivative nonlinear Schr\"odinger equation
\begin{equation*}
i\partial_tu+\partial^2_xu+i|u|^{2\sigma}\partial_xu=0,\ (t,x)\in\mathbb R\times\mathbb R
\end{equation*}
where $\sigma\in(0,1)$.
The equation has a two-parameter family of solitary waves
$$u_{\omega,c}(t,x)=\Phi_{\omega,c}(x)e^{i\omega t+\frac{ic}2x-\frac i{2\sigma+2}\int_0^x\Phi_{\omega,c}(y)^{2\sigma}dy},$$
with $(\omega,c)$ satisfying $\omega>c^2/4$, or $\omega=c^2/4$ and $c>0$.
The stability theory in the frequency region $\omega>c^2/4$  was studied previously. In this paper, we
prove the stability of the solitary wave solutions in the endpoint case $\omega=c^2/4$ and $c>0$.

\end{abstract}

\maketitle
\vskip0.5cm

\section{Introduction}

The derivative nonlinear Schr\"odinger (DNLS) equation
\begin{equation}\label{eq0}
i\partial_tv+\partial^2_xv+i\partial_x(|v|^2v)=0,\ (t,x)\in\mathbb R\times\mathbb R
\end{equation}
appears in the long wave-length approximation of Alfv\'en waves propagating in plasma \cite{m-o-m-t,mjolhus,m-m-w}.
Applying the gauge transformation
$$u(t,x)=v(t,x)e^{\frac i2\int_{-\infty}^x|v(t,x)|^2dx},$$
the equation \eqref{eq0} has the Hamiltonian form
\begin{equation}\label{eqdnls}
i\partial_tu+\partial^2_xu+i|u|^2\partial_xu=0,\ (t,x)\in\mathbb R\times\mathbb R.
\end{equation}

The Cauchy problem for \eqref{eqdnls} has been studied by many researchers.                          
The locally well-posedness theory  in the energy space  $H^1(\mathbb R)$
 was studied in \cite{hn,h-o92,h-o94,t-f}.
 Local well-posedness in low-regularity spaces $H^s(\mathbb R))$, $s\geq \frac12$ was investigated by Takaoka \cite{th}
 using the Fourier restricted method. Biagioni and Linares \cite{b-l}
proved that  when $s < \frac12$, the solution map from $H^s(\mathbb R)$ to $C([-T, T] : H^s(\mathbb R)),T>0$
for \eqref{eqdnls} is not locally uniformly continuous.

The problem of global well-posedness has attracted the attention of a number of authors.
Hayashi and Ozawa  \cite{h-o94,ot} proved the global existence in $H^1(\mathbb R)$ with $\|u_0\|_{L^2}^2<2\pi$.
Wu \cite{wyf13,wyf15} showed it holds for initial data $u_0$ having the mass $\|u_0\|_{L^2}^2$ less than threshold $4\pi$.
For the initial data with low regularity,
 Colliander, Keel,
Staffilani, Takaoka, and Tao \cite{c-k-s-t-t02} proved that the $H^s$-solution is global if $\|u_0\|_{L^2}^2<2\pi$
when $s > 1/2$ by the I-method (see also \cite{c-k-s-t-t01,th}).  Miao, Wu and Xu \cite{m-w-x}
showed that $H^{\frac12}$-solution is global if $\|u_0\|_{L^2}^2<2\pi$. Guo and Wu \cite{gzh-wyf} improved this result
to obtain that $H^{\frac12}$-solution is global if $\|u_0\|_{L^2}^2<4\pi$.
Despite the amount of studies devoted to \eqref{eqdnls}, existence of blowing up solutions remains an open problem.

 It is known that \eqref{eqdnls}
 has a two-parameter family of the solitary waves $u_{\omega,c}(t,x)=e^{i\omega t}\phi_{\omega,c}(x-ct)$, where
$(\omega,c)$ satisfies $\omega>c^2/4$ or $\omega=c^2/4$ and $c>0$.
 Boling Guo and Yaping Wu \cite{g-w} proved that the solitary waves $u_{\omega,c}$ are
orbitally stable when  $\omega>c^2/4$ and  $c<0$ by the abstract theory of Grillakis, Shatah,
and Strauss \cite{g-s-s87,g-s-s90} and the spectral analysis of the linearized operators. Colin and
Ohta \cite{c-o} proved that the solitary waves $u_{\omega,c}$ are orbitally stable when $\omega>c^2/4$ by
characterizing the solitary waves from the view point of a variational structure. The
case of $\omega=c^2/4$ and $c>0$ is treated by Kwon and Wu \cite{k-w}. Recently, the stability
of the multi-solitons is studied by Miao, Tang, and Xu \cite{m-t-x} and Le Coz and Wu \cite{c-w}.

Liu, Simpson, and Sulem \cite{l-s-s}
introduced an extension of \eqref{eqdnls} with general power nonlinearity, which is the so-called
generalized derivative nonlinear Schr\"odinger equation:
\begin{equation}\label{eqgdnls}
\begin{cases}
i\partial_tu+\partial^2_xu+i|u|^{2\sigma}\partial_xu=0,\ (t,x)\in\mathbb R\times\mathbb R\\
u(0,x)=u_0(x)
\end{cases}
\end{equation}
 where $\sigma>0$. The equation of \eqref{eqgdnls} is invariant under the scaling transformation
 $$u_\gamma(t,x)=\gamma^{\frac1{2\sigma}}u(\gamma^2t,\gamma x),\ \ \gamma>0,$$
which implies that its critical Sobolev exponent is $s_c=\frac12-\frac1{2\sigma}$.
Hayashi and Ozawa \cite{h-o16} proved local well-posedness in $H^1(\mathbb R)$ when $\sigma\geq1$ and showed that
the following quantities are conserved:
\begin{equation}\label{M}
M(u)=\|u\|_{L^2}^2=M(u_0),
\end{equation}
 \begin{align}\label{E}
E(u)=\|\partial_xu\|_{L^2}^2-\frac1{\sigma+1}Im\int_{\mathbb R}|u|^{2\sigma}u\overline{\partial_xu}dx=E(u_0),
\end{align}
\begin{equation}\label{P}
P(u)=Re\int_{\mathbb R}i\partial_xu\bar udx=Im\int_{\mathbb R}u\overline{\partial_xu}dx=P(u_0).
\end{equation}
Moreover, they proved global well-posedness for small initial data. They also constructed global solutions for any initial data
 in the $L^2$-subcritial case
$1/2\leq\sigma<1$.
Recently, Fukaya, Hayashi, Inui \cite{f-h-i} and Miao, Tang, xu \cite{m-t-x} investigate the global well-posedness for \eqref{eqgdnls} in the case $\sigma>1$ by  variational
argument.

Similar to the equation \eqref{eqdnls}, by \cite{l-s-s}, \eqref{eqgdnls} has a two-parameter family of solitary waves
$$u_{\omega,c}(t,x)=e^{i\omega t}\phi_{\omega,c}(x-ct),$$
with $(\omega,c)$ satisfying $\omega>c^2/4$, or $\omega=c^2/4$ and $c>0$,
\begin{align}\label{phi}
\phi_{\omega,c}(x)=\Phi_{\omega,c}(x)e^{\frac{ic}2x-\frac i{2\sigma+2}\int_0^x\Phi_{\omega,c}(y)^{2\sigma}dy}
\end{align}
and
\begin{align}\label{Phi}
\Phi_{\omega,c}(x)=
\begin{cases}
\left(\frac{(\sigma+1)(4\omega-c^2)}{2\sqrt{\omega}cosh(\sigma\sqrt{4\omega-c^2}x-c}\right)^{\frac1{2\sigma}},\ \ &if\ \omega>\frac{c^2}4,\\
\left(\frac{2(\sigma+1)c}{\sigma^2(cx)^2+1}\right)^{\frac1{2\sigma}},\ \ &if\ \omega=\frac{c^2}4 \ and\ c>0.
\end{cases}
\end{align}
Moreover, $\Phi_{\omega,c}$ is the positive even solution of
\begin{equation}\label{eqPhi}
-\partial^2_{x}\Phi+(\omega-\frac{c^2}4)\Phi+\frac c2|\Phi|^{2\sigma}\Phi-\frac{2\sigma+1}{(2\sigma+2)^2}|\Phi|^{4\sigma}\Phi,
\ \ x\in\mathbb R,
\end{equation}
and the complex-valued function $\phi_{\omega,c}$ satisfies
\begin{equation*}
-\partial^2_{x}\phi+\omega\phi+ic\partial_x\phi-i|\phi|^{2\sigma}\partial_x\phi=0
\ \ x\in\mathbb R.
\end{equation*}
Liu, Simpson and Sulem \cite{l-s-s} proved that  when $1<\sigma<2$, for some $z_0=z_0(\sigma)\in(0,1)$,  if
$-2\sqrt\omega<c<2z_0\sqrt\omega$,
the solitary waves are orbitally stable and if $2z_0\sqrt\omega<c<2\sqrt\omega$, they are orbitally unstable.
They also showed that the solitary waves for all $\omega>c^2/4$ are orbitally unstable when $\sigma\geq2$ and orbitally
stable when $0<\sigma<1$.
In \cite{fn}, it is proved that the solitary waves are orbitally unstable if $c=2z_0\sqrt\omega$ when $3/2<\sigma<2$.
Tang and Xu \cite{t-x}  investigated stability of the two sum of solitary waves for \eqref{eqgdnls}.
We also refer to \cite{hcc,sgn}
for some lower regularity results.

In this present work, we consider the stability of solitary wave solutions of \eqref{eqgdnls} in the endpoint case $\omega=c^2/4$
when $\sigma\in(0,1)$. For simplicity, we denote $\phi_{c}=\phi_{c^2/4,c}$, which solves
\begin{equation}\label{eqphi}
-\partial^2_{x}\phi+\frac {c^2}4\phi+ic\partial_x\phi-i|\phi|^{2\sigma}\partial_x\phi=0,\
\ \ x\in\mathbb R.
\end{equation}
 More precisely, we prove that when $\omega=c^2/4$, the solitary waves \eqref{phi}  are orbitally stable in the
  sense of the following.

\begin{theorem}\label{th1}
For any $\varepsilon>0$, there exists some $\delta=\delta(\varepsilon)$ such that if
\begin{equation}
\|u_0-\phi_c\|_{H^1}\leq\delta,
\end{equation}
then there exist $\theta(t)\in[0,2\pi), y(t)\in\mathbb R$
such that the solution $u(t)$ of the equation \eqref{eqgdnls} satisfies that,
 for any $t\in \mathbb R$,
$$\|u(t)-e^{i\theta(t)}\phi(\cdot-y(t))\|_{H^1}\leq\varepsilon.$$
\end{theorem}
We use a variational method to prove Theorem \ref{th1}, but it is not standard and the difficulty is from the ``zero mass" property of \eqref{eqphi} in this endpoint case.

The rest of the  paper is organized as follows. In section 2, we give a variational characterization of
solitary wave solutions. Then by a variational argument, we prove Theorem \ref{th1} in section 3.

\section{variational characterization}

In this section, we give a variational characterization of the solitary wave solution $\phi_c$ of \eqref{eqphi} defined by
\eqref{phi}.
Note that it is not standard because of the ``zero mass" of the equation \eqref{eqphi}. Our approach is inspired by Kwon, Wu \cite{k-w}.
In addition to \eqref{M}, \eqref{E} and \eqref{P}, we define several other variational functionals as follows:
\begin{align}\label{S}
S_c(u)(t)= &E(u)+cP(u)+\frac{c^2}4M(u)\\ \nonumber
=&\|\partial_xu\|_{L^2}^2+cIm\int_{\mathbb R}u\overline{\partial_xu}dx+\frac{c^2}4\|u\|_{L^2}^2-\frac1{\sigma+1}Im\int_{\mathbb R}|u|^{2\sigma}u\overline{\partial_xu}dx,
\end{align}
 \begin{equation}\label{K}
 K_c(u)(t)=\|\partial_xu\|_{L^2}^2+cIm\int_{\mathbb R}u\overline{\partial_xu}dx+\frac{c^2}4\|u\|_{L^2}^2-Im\int_{\mathbb R}|u|^{2\sigma}u\overline{\partial_xu}dx.
 \end{equation}
We also denote for convenience that
 \begin{equation}\label{L}
L_c(u)(t)=\|\partial_xu\|_{L^2}^2+cIm\int_{\mathbb R}u\overline{\partial_xu}dx+\frac{c^2}4\|u\|_{L^2}^2
 \end{equation}
and  \begin{equation}\label{N}
N(u)(t)=Im\int_{\mathbb R}|u|^{2\sigma}u\overline{\partial_xu}dx,
 \end{equation}
which imply that
$$S_c(u)=L_c(u)-\frac1{\sigma+1}N(u),\ \ K_c(u)=L_c(u)-N(u).$$

Using a standard argument as in  Berestycki and Lion  \cite{b-l}, we can obtain the uniqueness result as follows.
\begin{lemma}\label{lemuniq'}
If $\Psi\in H^1(\mathbb R)\setminus\{0\}$ is a solution of
\begin{equation}\label{eqPhi}
-\partial^2_{x}\Psi+\frac c2|\Psi|^{2\sigma}\Psi-\frac{2\sigma+1}{(2\sigma+2)^2}|\Psi|^{4\sigma}\Psi=0,
\ \ x\in\mathbb R,
\end{equation}
then there exists some $(\theta,x_0)$ such that $$\Psi(x)=e^{i\theta}\Phi_c(x-x_0)$$
with $$\Phi_c(x)=\left(\frac{2(\sigma+1)c}{\sigma^2(cx)^2+1}\right)^{\frac1{2\sigma}}.$$
\end{lemma}
The difficulty is that we have no $L^2$-control from $L_c(u)$. A counterpart result is as follows.
\begin{lemma}\label{lemuniq}
If $\psi\in H^1(\mathbb R)\setminus\{0\}$ is a solution of
\eqref{eqphi}
then there exists some $(\theta,x_0)$ such that $$\psi(x)=e^{i\theta}\phi_c(x-x_0).$$
\end{lemma}

We consider the following minimization problem:
\begin{align}\label{d}
d(c)=\inf\{S_c(u): u\in H^1(\mathbb R)\setminus\{0\}, K_c(u)=0\}.
\end{align}
Letting
$$ \tilde L_c(f):=L_c(e^{\frac c2ix}f),\ \tilde N_c(f):=N(e^{\frac c2ix}f),$$
then
$$ \tilde L_c(f)=\|\partial_xf\|_{L^2}^2,\ \tilde N_c(f)
=-\frac c2\|f\|_{L^{2\sigma+2}}^{2\sigma+2}+ Im\int_{\mathbb R}
|f|^{2\sigma}f\overline{\partial_x f}dx,$$ and
$$\tilde S_c(f):=S_c(e^{\frac c2ix}f)=\tilde L_c(f)-\frac1{\sigma+1}\tilde N_c(f),$$
$$\tilde K_c(f):=K_c(e^{\frac c2ix}f)=\tilde L_c(f)-\tilde N_c(f).$$
Hence, equivalently, we have
\begin{align}\label{dtilde}
d(c)=\inf\{\tilde S_c(v): v\in H^1(\mathbb R)\setminus\{0\}, \tilde K_c(v)=0\}.
\end{align}

We need the following result to give the characterization of $d(c)$.
\begin{lemma}\label{lemkgeq0}
Assume that $f\in H^1(\mathbb R)\setminus\{0\}$ satisfies $$\|\partial_xf\|_{L^2(\mathbb R)}^2\leq\frac{\sigma+1}\sigma d(c),$$
then $\tilde K_c(f)\geq0.$
\end{lemma}
\begin{proof}
We argue by contradiction to assume that there exists some $f\in H^1(\mathbb R)\setminus\{0\}$ such that $\tilde K_c(f)<0$.
Then, there exists some $\gamma\in(0,1)$ such that $\tilde K_c(f_\gamma)=0$, with $f_\gamma(x)=\gamma^{\frac1\sigma}f(\gamma x)$.
Indeed, since
$$\tilde K_c(f_\gamma)=\gamma^{\frac2\sigma+1}\left[\|\partial_xf\|_{L^2}^2+\frac c2\|f\|_{L^{2\sigma+2}}^{2\sigma+2}-\gamma Im\int_{\mathbb R}
|f|^{2\sigma}f\overline{\partial_x f}dx\right],$$
$\tilde K_c(f)<0$ implies that we may choose $$\gamma=\frac{\|\partial_xf\|_{L^2}^2+\frac c2\|f\|_{L^{2\sigma+2}}^{2\sigma+2}}{Im\int_{\mathbb R}
|f|^{2\sigma}f\overline{\partial_x f}dx}<1$$
to get $\tilde K_c(f_\gamma)=0$.
Therefore, by definition of $d(c)$ in \eqref{dtilde}, $\tilde S_c(f_\gamma)\geq d(c)$, which gives then
$$\|\partial_xf_\gamma\|_{L^2}^2=\frac{\sigma+1}\sigma\left(\tilde S_c(f_\gamma)-\frac1{\sigma+1}\tilde K_c(f_\gamma)\right)\geq\frac{\sigma+1}\sigma d(c),$$
or $ \gamma^{\frac2\sigma+1}\|\partial_xf\|_{L^2}^2\geq\frac{\sigma+1}\sigma d(c).$
Since $\gamma<1$, this contradicts the assumption $\|\partial_xf\|_{L^2(\mathbb R)}^2\leq\frac{\sigma+1}\sigma d(c)$
and we conclude the proof.
\end{proof}

Now we give the characterization of $d(c)$.
\begin{lemma}\label{lemSphi}
It holds that
$$S_c(\phi_c)=d(c)$$
with $\phi_c$ defined by \eqref{phi} with $w=c^2/4$.
\end{lemma}
\begin{proof}
We first claim that $$d(c)>0.$$

Indeed, considering
\begin{align*}
\tilde S_c(v)-\frac1{\sigma+1}\tilde K_c(v)
=(1-\frac1{\sigma+1})\tilde L_c(v)
=(1-\frac1{\sigma+1})\|\partial_xv\|_{L^2}^2\geq0.
\end{align*}
Now, if $d(c)=0$, then we obtain some minimizing sequence $\{v_n\}\subset H^1(\mathbb R)\setminus\{0\}$, such that
$$\tilde S_c(v_n)\rightarrow0, \ \ and\ \ \tilde K_c(v_n)=0.$$
Thus, we have $\|\partial_xv_n\|_{L^2}^2\rightarrow0$ and
$$\tilde N_c(v_n)\rightarrow0.$$
that is
\begin{equation}\label{e1}
\int_{\mathbb R}\left(\frac c2|v_n|^{2\sigma+2}-Im|v_n|^{2\sigma}v_n\overline{\partial_xv_n}\right)dx\rightarrow0.
\end{equation}
Note that
\begin{align}\label{e2}
&\left|Im\int_{\mathbb R}|v_n|^{2\sigma}v_n\overline{\partial_xv_n}dx\right|
\leq\|\partial_xv_n\|_{L^2}\|v_n\|^{2\sigma+1}_{L^{4\sigma+2}}\\ \nonumber
&\leq\|\partial_xv_n\|_{L^2}^\theta\cdot\|v_n\|^{2\sigma+2-\theta}_{L^{2\sigma+2}}
\leq\frac c4\|v_n\|^{2\sigma+2}_{L^{2\sigma+2}}+A\|\partial_xv_n\|_{L^2}^{2\sigma+2}
\end{align}
with some $\theta\in(1,2\sigma+2)$ and some constants $A>0$.
From \eqref{e1},\eqref{e2} and $\|\partial_xv_n\|_{L^2}^2\rightarrow0$ ,  we obtain that
$$\|v_n\|^{2\sigma+2}_{L^{2\sigma+2}}\rightarrow0$$
which gives by interpolation that
\begin{equation}\label{infty0}
\|v_n\|_{L^\infty}\rightarrow0,\ as\ n\rightarrow\infty.
\end{equation}
Using \eqref{infty0}, we obtain that
\begin{align}\label{2sigma+2}
0=K_c(u_n)&=\|\partial_xv_n\|_{L^2}^2+\frac c2\int_{\mathbb R}|v_n|^{2\sigma+2}dx
-Im\int_{\mathbb R}|v_n|^{2\sigma}v_n\overline{\partial_xv_n}dx\\ \nonumber
&\geq\|\partial_xv_n\|_{L^2}^2+\frac c2\int_{\mathbb R}|v_n|^{2\sigma+2}dx
-\frac12\|\partial_xv_n\|_{L^2}^2-\frac 12\int_{\mathbb R}|v_n|^{4\sigma+2}dx\\ \nonumber
&=\frac12\|\partial_xv_n\|_{L^2}^2+\int_{\mathbb R}|v_n|^{2\sigma+2}(\frac c2-\frac12|v_n|^{2\sigma})dx\\ \nonumber
&\geq\frac12\|\partial_xv_n\|_{L^2}^2+\frac c4\int_{\mathbb R}|v_n|^{2\sigma+2}\geq0,
\end{align}
which gives $v_n=0$. This is a contradiction and gives the claim $d(c)>0$.

Next, let  $\{v_n\}\subset H^1(\mathbb R)\setminus\{0\}$ be the minimizing sequence such that as  $n\rightarrow\infty$,
$$\tilde S_c(v_n)=\|\partial_xv_n\|_{L^2}^2+\frac1{\sigma+1}\left(\frac c2\int_{\mathbb R}|v_n|^{2\sigma+2}dx
-Im\int_{\mathbb R}|v_n|^{2\sigma}v_n\overline{\partial_xv_n}dx\right)\rightarrow d(c),$$ and
$$\tilde K_c(v_n)=\|\partial_xv_n\|_{L^2}^2+\frac c2\int_{\mathbb R}|v_n|^{2\sigma+2}dx
-Im\int_{\mathbb R}|v_n|^{2\sigma}v_n\overline{\partial_xv_n}dx=0.$$
There holds then $\|\partial_xv_n\|_{L^2}^2\rightarrow\frac{\sigma+1}\sigma d(c)$.
Moreover, by a similar argument of \eqref{2sigma+2},  there exists some absolute constant $C>0$ such that
$$\|v_n\|_{L^{2\sigma+2}}\leq C.$$
Now we apply the profile decomposition to the uniformly bounded sequence $\{v_n\}$ in $\dot{H}^1(\mathbb R)\bigcap L^{2\sigma+2}(\mathbb R)$
to obtain that there exist some sequences $\{V^j\}_{j=1}^\infty$ and $\{x_n^j\}_{n,j=1}^\infty$ such that,
up to some supsequence, for each $L\geq1$,
\begin{align}
v_n=\Sigma_{j=1}^{L}V^j(\cdot-x_n^j)+R_n^L
\end{align}
with $$\forall k\neq j,\ \ |x_n^j-x_n^k|\rightarrow\infty,\ \ as\ \ n\rightarrow\infty$$
and
\begin{align}\label{orth1}
\lim_{L\rightarrow\infty}\left[\lim_{n\rightarrow\infty}\|R_n^L\|_{L^q(\mathbb R)}\right]=0,\  \ \forall q>2\sigma+2.
\end{align}
Moreover,
\begin{align}\label{orth2}
\|\partial_xv_n\|_{L^2}^2=\Sigma_{j=1}^{L}\|\partial_xV^j\|_{L^2}^2+\|\partial_xR_n^L\|_{L^2}^2+o_n(1),
\end{align}
\begin{align}
\tilde S_c(v_n)=\Sigma_{j=1}^{L}\tilde S_c(V^j)+\tilde S_c(R_n^L)+o_n(1),
\end{align}
and
\begin{align}\label{orth3}
\tilde K_c(v_n)=\Sigma_{j=1}^{L}\tilde K_c(V^j)+\tilde K_c(R_n^L)+o_n(1).
\end{align}
Thus by Lemma \ref{lemkgeq0}, $\tilde K_c(V^j)=K_c(e^{\frac c2ix}V^j)\geq0$ or $V^j=0$.
Since also $\tilde K_c(R_n^L)\geq0$, from \eqref{orth3} and $\tilde K(v_n)=0$, we get that
for any $j=1,2,\cdots,L$ there must hold that $\tilde K_c(V^j)=0,$ which means that $\tilde S_c(V^j)\geq d(c)$ or $V^j=0$.
Now from \eqref{orth2}, there exists only one $j$, say $j=1$, such that
$\tilde S_c(V^1)=d(c)$ and $V^j=0$ for $j=2,\cdots, L$.
Hence, $V^1$ is the function such that
$$\tilde S_c(V^1)=d(c),\ \ \tilde K_c(V^1)=0.$$
Then there exists some Lagrange constant $\rho$ such that
$$\tilde S'_c(V^1)=\rho \tilde K'_c(V^1),$$ which implies
$$\langle \tilde S'_c(V^1),V^1\rangle=\rho\langle \tilde K'_c(V^1),V^1\rangle.$$
That is to say
$$(\rho(\sigma+1)-1)\tilde K_c(V^1)=\rho\sigma\|\partial_xV^1\|_{L^2}^2,$$ which implies that $\rho=0$.
Therefore, we have that $$\tilde S'_c(V^1)=0,\ \ or\ \ S'_c(e^{\frac c2ix}V^1)=0$$
and then $e^{\frac c2ix}V^1(x)$ solves the equation \eqref{eqphi}.
Hence by Lemma \ref{lemuniq},
$$e^{\frac c2ix}V^1(x)=e^{i\theta}\phi_c(x-x_0).$$
It follows that $$d(c)=\tilde S_c(V^1)=S_c(e^{i\theta}\phi_c(\cdot-x_0))=S_c(\phi_c).$$
\end{proof}

Finally in this section, we prove the following lemma, which is useful to show our main result.
\begin{lemma}\label{lemS''<0}\footnote{Due to a private discussion with Cui Ning.}
There hold that
\begin{align}\label{S''}
P(\phi_c)+\frac c2M(\phi_c)>0, \ \ \partial_cP(\phi_c)+\frac c2\partial_cM(\phi_c)>0.
\end{align}
\end{lemma}
\begin{proof}
First of all,  by the definition of \eqref{phi} and \eqref{Phi} and  straight calculation, it follows that
\begin{align}\label{c1}
P(\phi_c)+\frac c2M(\phi_c)
 =\frac1{2\sigma+2}\|\Phi_c\|_{L^{2\sigma+2}}^{2\sigma+2}>0.
\end{align}

Next, letting $\phi:=\phi_1$,  it follows that $\phi_c(x)=c^{\frac1{2\sigma}}\phi(cx)$,
$$\int|\phi_c|^2dx=c^{\frac1{\sigma}-1}\|\phi\|_{L^2}^2,\ \ P(\phi_c)=c^{\frac1{\sigma}}P(\phi).$$
Hence,
\begin{align*}
 \partial_cP(\phi_c)+\frac c2\partial_cM(\phi_c)
 =c^{\frac1{\sigma}-1}\left(\frac1\sigma P(\phi)+\frac12(\frac1\sigma-1)M(\phi)
 \right).
\end{align*}
Note that
$$P(\phi_c)=-\frac12c\|\Phi_c\|_{L^2}^2+\frac1{2\sigma+2}\|\Phi_c\|^{2\sigma+2}_{2\sigma+2}$$
and recall
$$\Phi(x):=\Phi_1(x)=\left(\frac{2(\sigma+1)}{(\sigma x)^2+1}\right)^{\frac1{2\sigma}}.$$
We have then
\begin{align*}
&\frac1\sigma P(\phi)+\frac12(\frac1\sigma-1)M(\phi)\\
=&-\frac1{2\sigma}\int_{\mathbb R}\Phi^2dx+\frac1{2\sigma(\sigma+1)}\int_{\mathbb R}\Phi^{2\sigma+2}dx+(\frac1{2\sigma}-\frac12)\int_{\mathbb R}\Phi^2dx\\
=&\frac1{2\sigma(\sigma+1)}\int_{\mathbb R}\Phi^{2\sigma+2}dx-\frac12\int_{\mathbb R}\Phi^2dx\\
=&\frac{(2\sigma+2)^{\frac{1}\sigma}}{2}\int_{\mathbb R}
\left(\frac{2}{\sigma(\sigma^2x^2+1)^{\frac{\sigma+1}{\sigma}}}
-\frac{1}{(\sigma^2x^2+1)^{\frac{1}{\sigma}}}
\right)dx\\
=&\frac{(2\sigma+2)^{\frac{1}\sigma}}{2\sigma}\int_{\mathbb R}
\left(\frac{2}{\sigma(x^2+1)^{\frac{\sigma+1}{\sigma}}}
-\frac{1}{(x^2+1)^{\frac{1}{\sigma}}}
\right)dx.
\end{align*}

Finally, we are sufficed to show the integration
\begin{align}\label{c2}I=I(\sigma)=\int_{\mathbb R}
\left(\frac{2}{\sigma(x^2+1)^{\frac{\sigma+1}{\sigma}}}
-\frac{1}{(x^2+1)^{\frac{1}{\sigma}}}
\right)dx>0.
\end{align}
In fact, by straight calculation,
\begin{align*}
\frac d{dx}\left[x(x^2+1)^{-\frac1\sigma}\right]
&=(x^2+1)^{-\frac1\sigma}-\frac1\sigma\cdot 2x^2\cdot(x^2+1)^{-\frac1\sigma-1}\\
&=(1-\frac2\sigma)(x^2+1)^{-\frac1\sigma}+\frac2\sigma(x^2+1)^{-\frac1\sigma-1},
\end{align*}
which implies that if $\sigma<2$,
\begin{align*}
\int_{\mathbb R}\left[(1-\frac2\sigma)(x^2+1)^{-\frac1\sigma}+\frac2\sigma(x^2+1)^{-\frac1\sigma-1}\right]dx=0
\end{align*}
or
\begin{align*}
\int(x^2+1)^{-\frac1\sigma-1}dx=(1-\frac\sigma2)\int(x^2+1)^{-\frac1\sigma}dx.
\end{align*}
Since $\sigma<1$, we obtain that
\begin{align*}
I(\sigma)=\int_{\mathbb R}
\left(\frac{2}{\sigma(x^2+1)^{\frac{\sigma+1}{\sigma}}}
-\frac{1}{(x^2+1)^{\frac{1}{\sigma}}}
\right)dx=2(\frac1\sigma-1)\int(x^2+1)^{-\frac1\sigma}dx>0.
\end{align*}

As a result, we conclude Lemma \ref{lemS''<0} by \eqref{c1} and \eqref{c2}.
\end{proof}

\section{Proof of Theorem \ref{th1}}
In this section, we prove Theorem \ref{th1}. The proof is based on the variational characterization of solitary
wave solutions in section 2. Using the notations defined in the above section, we set
$$\mathcal A^+=\{u\in H^1(\mathbb R)\setminus\{0\}: S_c(u)<S_c(\phi_c), L_c(u)<\frac{\sigma+1}\sigma d(c)\},$$
$$\mathcal A^-=\{u\in H^1(\mathbb R)\setminus\{0\}: S_c(u)<S_c(\phi_c), L_c(u)>\frac{\sigma+1}\sigma d(c)\}.$$

\begin{lemma}\label{leminva}
The sets $\mathcal A^+$ and $\mathcal A^-$ are invariant under the flow of \eqref{eqgdnls}, i.e., if $u_0\in\mathcal A^+$ (resp. $\mathcal A^-$),
then the solution $u(t)$ of \eqref{eqgdnls} with $u(0)=u_0$ belongs to $\mathcal A^+$ (resp. $\mathcal A^-$) as long as $u(t)$ exists.
\end{lemma}
\begin{proof}
Let $u_0\in\mathcal A^+$ and $I=(-T_*,T^*)$ be the maximal existence interval of the solution $u(t)$ of \eqref{eqgdnls} with $u(0)=u_0$.
By $u_0\neq0$ and the conservation laws \eqref{E}, \eqref{M} and \eqref{P}, we have that $u(t)\neq0$ for $t\in I$. By definition of $S_c(u)$,
$S_c$ is also conserved, which means that $S_c(u(t))=S_c(u_0)<d(c)$ for $t\in I$.
By continuity of the function $t\mapsto L_c(u(t))$, we assume that there exits some $t_0$ such that
$$L_c(u(t_0))=\tilde L_c(e^{-\frac c2ix}u(t_0))=\frac{\sigma+1}\sigma d(c).$$
Thus, from
\begin{equation*}
d(c)>S_c(u(t_0))=\frac\sigma{\sigma+1}L_c(u(t_0))+\frac1{\sigma+1}K_c(u(t_0))=d(c)+\frac1{\sigma+1}K_c(u(t_0)),
\end{equation*}
we get that
$\tilde K_c(e^{-\frac c2ix}u(t_0))=K_c(u(t_0))<0$.
On the other hand, by Lemma \ref{lemkgeq0}, it holds that $$\tilde L_c(e^{-\frac c2ix}u(t_0))=L_c(u(t_0))>\frac{\sigma+1}\sigma d(c)$$
 and we get a contradiction.
Hence $\mathcal A^+$ is invariant under the flow
of \eqref{eqgdnls}.

In the same way, we see that $\mathcal A^-$ is invariant under the flow
of \eqref{eqgdnls}.
\end{proof}

\begin{lemma}\label{lem2}
Let $w\in H^1(\mathbb R)$. For any $\varepsilon>0$, there exists some $\delta>0$ such that
if
\begin{align}
|S_c(w)-S_c(\phi_c)|+|K_c(w)|<\delta,
\end{align}
then
\begin{align*}
\inf_{(\theta,y)\in\mathbb R^2}\|w-e^{i\theta}\phi_c(\cdot-y)\|_{\dot H^1}<\varepsilon.
\end{align*}
\end{lemma}
\begin{proof}
By contradiction, we assume that there exist $\varepsilon_0>0$ and some sequences $\{w_n\}\subset H^1(\mathbb R)$
such that
$$S_c(w_n)\rightarrow d(c),\ \ K_c(w_n)\rightarrow0,\ \ as\ n\rightarrow\infty,$$ but
\begin{align}\label{varepsilon0}
\inf_{(\theta,y)\in\mathbb R^2}\|w_n-e^{i\theta}\phi_c(\cdot-y)\|_{\dot H^1}>\varepsilon_0.
\end{align}

On the other hand, we
follow the
  the proof of Lemma \ref{lemSphi}, using the profile decomposition, to find that
$$e^{-\frac c2ix}w_n-V^1(\cdot-x_n^1)\rightarrow0,\ \ in\ \dot H^1(\mathbb R)$$
i.e. $$w_n(\cdot+x_n^1)-e^{\frac c2ix}V^1\rightarrow0,\ \ in\ \dot H^1(\mathbb R)$$
and
$e^{\frac c2ix}V^1(x)=e^{i\theta}\phi_c(x-x_0)$ solves the equation \eqref{eqphi}.
Hence, we obtain that
for large $n$ it holds that
\begin{align*}
\inf_{(\theta,y)\in\mathbb R^2}\|w_n-e^{i\theta}\phi_c(\cdot-y)\|_{\dot H^1}
\leq \|w_n-e^{i\theta_0}\phi_c(\cdot-x_0-x_n)\|_{\dot H^1}
<\varepsilon_0,
\end{align*}
which is a contradiction with \eqref{varepsilon0}.
\end{proof}

\begin{lemma}\label{lem3}
For any $\varepsilon>0$, there exists some $\delta>0$ such that if $\|u_0-\phi_c\|_{H^1}<\delta$,
then the solution $u(t)$ of \eqref{eqgdnls} with $u(0)=u_0$ satisfies
\begin{align*}
|S_c(u(t))-S_c(\phi_c)|+|K_c(u(t))|<\varepsilon,\ \ for\ \ t\in I
\end{align*}
where $I$ is the maximal lifespan.
\end{lemma}
\begin{proof}
For sufficiently small $\delta>0$, which will be determined  later, it follows that $\|u_0-\phi_c\|_{H^1}<\delta$
implies
\begin{align}\label{delta1}
S_{c+\lambda}(u_0)=S_{c+\lambda}(\phi_c)+O(\delta)
\end{align}
where $\lambda$ will be chosen later such that $|\lambda|$ is small.
Applying Taylor expansion to the function $S_{c+\lambda}(\phi_{c+\lambda})$ of $\lambda$, we obtain that
\begin{align*}
S_{c+\lambda}(\phi_{c+\lambda})&=E(\phi_{c+\lambda})+(c+\lambda)P(\phi_{c+\lambda})+\frac{(c+\lambda)^2}4M(\phi_{c+\lambda})\\
&=S_{c}(\phi_{c+\lambda})+\lambda(P(\phi_{c+\lambda})+\frac{c}2M(\phi_{c+\lambda}))+\frac{\lambda^2}4M(\phi_{c+\lambda})\\
&=S_{c}(\phi_{c+\lambda})-S_{c}(\phi_{c})+\lambda\left(P(\phi_{c+\lambda})-P(\phi_c)+\frac{c}2M(\phi_{c+\lambda})-\frac{c}2M(\phi_c)\right)\\
&+\frac{\lambda^2}4\left(M(\phi_{c+\lambda})-M(\phi_{c})\right)+S_{c}(\phi_{c})+
\lambda\left(P(\phi_c)+\frac{c}2M(\phi_c)\right)+\frac{\lambda^2}4M(\phi_{c})\\
&=\frac{\lambda^2}2\langle S_c''(\phi_c)\partial_c\phi_c,\partial_c\phi_c\rangle
+\lambda^2\left(\langle P'(\phi_c),\partial_c\phi_c\rangle+\frac c2\langle M'(\phi_c),\partial_c\phi_c\rangle
\right)\\&+S_{c+\lambda}(\phi_{c})+o(\lambda^2)\\
&=-\frac{\lambda^2}2\langle S_c''(\phi_c)\partial_c\phi_c,\partial_c\phi_c\rangle+S_{c+\lambda}(\phi_{c})+o(\lambda^2),
\end{align*}
where we use the formular
$$\langle S_c''(\phi_c)\partial_c\phi_c,\partial_c\phi_c\rangle=-\langle P'(\phi_c),\partial_c\phi_c\rangle-\frac c2\langle M'(\phi_c),\partial_c\phi_c\rangle
=-\partial_cP(\phi_c)-\frac c2\partial_cM(\phi_c)$$
which is negative by Lemma \ref{lemS''<0}.
Thus, combined with \eqref{delta1}, we obtain that
\begin{align*}
S_{c+\lambda}(u_0)=S_{c+\lambda}(\phi_{c+\lambda})+\frac{\lambda^2}2\langle S_c''(\phi_c)\partial_c\phi_c,\partial_c\phi_c\rangle+o(\lambda^2)+O(\delta).
\end{align*}
For any $\lambda$ satisfying $|\lambda|\in(0,\lambda_0)$ with some  $\lambda_0>0$ small enough, we may choose $\delta>0$ small such that
\begin{align}\label{lambda1}
S_{c+\lambda}(u_0)<S_{c+\lambda}(\phi_{c+\lambda}).
\end{align}

Now we deal with $L_c$. Note that
\begin{align*}
L_{c+\lambda}(u_0)=L_{c+\lambda}(\phi_{c})+O(\delta).
\end{align*}
By Taylor expansion, we estimate
\begin{align*}
L_{c+\lambda}(u_0)&=L_{c+\lambda}(\phi_{c})+O(\delta)\\
&=L_{c+\lambda}(\phi_{c})-L_{c}(\phi_{c})+L_{c}(\phi_{c})+O(\delta)\\
&=\lambda\left(P(\phi_{c})+\frac{c}2M(\phi_c)
\right)+\frac{\sigma+1}\sigma d(c)+o(\lambda)+O(\delta)\\
&=\frac{\sigma+1}\sigma d(c+\lambda)-\frac{\sigma+1}\sigma (d(c+\lambda)-d(c))+\lambda\left(P(\phi_{c})+\frac{c}2M(\phi_c)
\right)+o(\lambda)+O(\delta)\\
&=\frac{\sigma+1}\sigma d(c+\lambda)-\frac{\lambda}{\sigma}\left(P(\phi_{c})+\frac{c}2M(\phi_c)
\right)+o(\lambda)+O(\delta),
\end{align*}
which, by choosing $\lambda_0$ and $\delta>0$ smaller, is small than $\frac{\sigma+1}\sigma d(c+\lambda)$,
i.e.
\begin{align*}
L_{c+\lambda}(u_0)<\frac{\sigma+1}\sigma d(c+\lambda),\ \ \forall\lambda\in(0,\lambda_0).
\end{align*}
Similarly, it must hold that for small $\lambda_0$ and $\delta>0$,
\begin{align*}
L_{c-\lambda}(u_0)>\frac{\sigma+1}\sigma d(c-\lambda),\ \ \forall\lambda\in(0,\lambda_0).
\end{align*}
In view of the invariant sets $\mathcal A^\pm$, these estimates, combined with \eqref{lambda1}, imply that
for any $t\in I=(-T_*,T^*)$,
\begin{align}\label{lambda2}
L_{c+\lambda}(u(t))<\frac{\sigma+1}\sigma d(c+\lambda),\ \
L_{c-\lambda}(u(t))>\frac{\sigma+1}\sigma d(c-\lambda).
\end{align}
Therefore,  we can choose sufficiently small $\lambda_0>0$ and $\delta>0$ such that
if  $\|u_0-\phi_c\|_{H^1}<\delta$,
\begin{align*}
|S_c(u(t))-S_c(\phi_c)|+|L_c(u(t))-L_c(\phi_c)|<\frac{\varepsilon}{A},\ \ for\ \ t\in I
\end{align*}
with some constant $A>0$ large enough.
Since $K_c=(\sigma+1)S_c-\sigma L_c$, we finally obtain that
\begin{align*}
|S_c(u(t))-S_c(\phi_c)|+|K_c(u(t))|<\varepsilon,\ \ for\ \ t\in I.
\end{align*}
\end{proof}


Now we prove Theorem \ref{th1}.\\
{\bf The  Proof Theorem 1.1.}
By contradiction, we assume that there exists some $\varepsilon_0>0$ such that for any small $\delta>0$ there exists some
sequence $\{t_n\}$ satisfying $\|u_0-\phi_c\|_{H^1}<\delta$, but
\begin{equation}\label{3.6}
\inf_{(\theta,y)\in\mathbb R^2}\|u(t_n)-e^{i\theta}\phi_c(\cdot-y)\|_{H^1}^2\geq\varepsilon_0.
\end{equation}
Taking $\delta$ small enough, we get from Lemma \ref{lem3} that
\begin{align*}
|S_c(u(t_n))-S_c(\phi_c)|+|K_c(u(t_n))|<\tilde\delta,
\end{align*}
where $\tilde\delta>0$ can be chosen sufficiently small such that, using Lemma \ref{lem2}, it holds that
\begin{equation*}
\inf_{(\theta,y)\in\mathbb R^2}\|u(t_n)-e^{i\theta}\phi_c(\cdot-y)\|_{\dot H^1}^2<\frac{\varepsilon_0}2.
\end{equation*}
Moreover, by mass conversation, we have
$$\|u(t_n)\|_{L^2}^2-\|\phi_c\|_{L^2}^2=O(\delta).$$
Choosing $\delta$ smaller, we obtain then
$$\inf_{(\theta,y)\in\mathbb R^2}\|u(t_n)-e^{i\theta}\phi_c(\cdot-y)\|_{L^2}^2<\frac{\varepsilon_0}2.$$
Finally, there holds that
\begin{equation*}
\inf_{(\theta,y)\in\mathbb R^2}\|u(t_n)-e^{i\theta}\phi_c(\cdot-y)\|_{ H^1}^2<\varepsilon_0,
\end{equation*}
which  contradicts \eqref{3.6}.
\  \ \ \ \ \ \ \ \ \ \ \  \ \ \ \ \ \ \ \ \ \ \  \ \ \ \ \ \ \ \ \ \ \  \ \ \ \ \ \ \ \ \ \ \  \ \ \ \ \ \ \ \ \ \ \ $\Box$

\vspace{0.5cm}

{\bf Acknowledgment }
The author would like to thank Cui Ning for useful discussions.
This work  was  supported by the National Natural Science Foundation of China (No.11301564, 11771469).

\vspace{0.5cm}


\end{document}